\documentclass[a4paper]{amsart}
\usepackage[utf8]{inputenc}  
\usepackage{currfile}
\usepackage{graphicx}
\usepackage{upref}
\usepackage{url}
\usepackage[pdftex, colorlinks=true]{hyperref}
\usepackage{color}
\definecolor{darkred}{rgb}{0.6,0,0}

\newcommand{\abs}[1]{\left\vert#1\right\vert}
\newcommand{\R}{\mathbb R}

\newcommand{\norm}[1]{\left\Vert#1\right\Vert}
\newcommand{\seq}[1]{\left\{#1\right\}}
\newcommand{\Dp}{D_+}
\newcommand{\test}{\varphi}

\DeclareMathOperator*{\sgn}{sign}
\newcommand{\sign}[1]{\sgn\left(#1\right)}
\newcommand{\dott}{\, \cdot\,}
\newcommand{\cha}[1]{\chi_{#1}}
\newcommand{\marginlabel}[1]%
       {\mbox{}\marginpar{\raggedleft\hspace{0pt}\tiny{\textcolor{red}{#1}}}}

\newcommand{\arxiv}[1]{\href{http://arxiv.org/pdf/#1}{arXiv:#1}} 

\newtheorem{theorem}{Theorem}[section]

\newtheorem{lemma}[theorem]{Lemma}

\newtheorem{remark*}[theorem]{Remark}

\numberwithin{equation}{section}     
\allowdisplaybreaks

\begin{document}
\title[The limit of Follow-the-Leader
to LWR]{The continuum limit of Follow-the-Leader
models --- a short proof}

\author[Holden]{Helge Holden}
\address[Holden]{\newline
    Department of Mathematical Sciences,
    NTNU Norwegian University of Science and Technology,
    NO--7491 Trondheim, Norway}
\email[]{\href{helge.holden@ntnu.no}{helge.holden@ntnu.no}} 
\urladdr{\href{https://www.ntnu.edu/employees/holden}{https://www.ntnu.edu/employees/holden}}

\author[Risebro]{Nils Henrik Risebro}
\address[Risebro]{\newline
 Department of Mathematics,
University of Oslo,
  P.O.\ Box 1053, Blindern,
  NO--0316 Oslo, Norway }
\email[]{\href{nilshr@math.uio.no}{nilshr@math.uio.no}} 

\date{\today} 



\thanks{Research was supported by the grant {\it Waves and Nonlinear Phenomena (WaNP)} from the Research Council of Norway. The research was done while the authors were at Institut Mittag-Leffler, Stockholm.}

\dedicatory{We dedicate this paper to the memory of Hans Petter Langtangen (1962--2016)}

\begin{abstract} 
We offer a simple and self-contained proof that the  Follow-the-Leader model converges to the Lighthill--Whitham--Richards model for traffic flow.
\end{abstract}

\maketitle
\section{Introduction} \label{sec:intro}
The problem of convergence of particle models to continuum models is fundamental. We here study it in the context of traffic flow. In this case there are two fundamentally different models: The first one is based on  individual vehicles whose dynamics is determined by the 
behavior of the vehicle immediately in front of it. This gives the Follow-the-Leader (FtL) model, which constitutes a system of ordinary differential equations describing the dynamics of individual vehicles. The other model is based on the assumption of heavy traffic where the individual vehicles are represented by a density. Assuming that the number of vehicles is conserved, we get the classical Lighthill--Whitham--Richards (LWR) model \cite{LW_II, richards}, which is nothing but a scalar hyperbolic conservation law.  The question that we address in this paper is in what sense the FtL model approaches or approximates the LWR model in the case of dense traffic. 

The principal assumption in FtL models is that the velocity $V$ of any
given vehicle is a function of the distance to the vehicle in front of it. We
shall write this function as
\begin{equation*}
  V\left(\frac{\Delta Z}{\ell}\right),
\end{equation*}
where $\Delta Z$ denotes the distance to nearest vehicle in front, and $\ell$ the length of each
vehicle. For obvious reasons, $\Delta Z\ge \ell$. 
It is commonly assumed that $V$ is an increasing positive function
defined in $[1,\infty)$, such that $\lim_{y\to\infty} V(y)=v_{\max}<\infty$.
Consider $N$ vehicles with length $\ell$ and position
$Z_1(t)<\dots<Z_N(t)$ on the real axis with dynamics given by
\begin{equation}\label{eq:Zeq}
  \frac{d}{dt}Z_i=V\Bigl(\frac{Z_{i+1}-Z_i}{\ell}\Bigr) \ \text{ for $i=1,\dots, N-1$.}
\end{equation}
To close this system, we must prescribe the velocity of the first vehicle
at $Z_N$. It is natural to model this by letting $\dot{Z}_N=v_{\max}$.

In this paper we analyze the limit of this system of ordinary
differential equations when $N\to\infty$ and $\ell\to0$. 
We show that  
\begin{equation*}
  \frac{\ell}{Z_{i+1}(t)-Z_i(t)} \to \rho(t,z), 
\end{equation*}
where intuitively  $Z_{i+1},Z_i\to z$, and where $\rho$ is an entropy solution to the scalar conservation law
\begin{equation}\label{eq:scalcons}
  \rho_t + f(\rho)_z = 0, \quad f(\rho)=\rho V\Bigl(\frac{1}{\rho}\Bigr).
\end{equation}
This problem has also been addressed by several other researchers. We here mention \cite{Argall_etal,AwKlarMaterneRascle,CristianiSahu,1605.05883,GoatinRossi,1211.4619}. 
The long and technically demanding paper  \cite{FrancescoRosini} shows
this convergence, while
in \cite{ColomboRossi,Rossi}, the convergence of the discrete system is assumed rather than proved.
The approach here resembles \cite{1702.01718} where  FtL models are viewed as a numerical approximation of the LWR model, and the proof of convergence depends on classical results by Crandall–Majda and Wagner for a grid approximation.

Here we offer is a simple and straightforward proof of the continuum limit.

Solutions to scalar conservation laws are in general not
continuous, and \eqref{eq:scalcons} must be considered in the weak
sense; furthermore weak solutions to the Cauchy problem are not
unique, and in order for the Cauchy problem to have a unique solution,
one must impose the Kru\v{z}kov entropy condition \cite{HoldenRisebro}:
A function $\rho\in C([0,\infty);L^1(\R))$ is called an \emph{entropy
  solution} to the Cauchy problem for \eqref{eq:scalcons} if for all
constants $k\in \R$ and all non-negative test functions $\test\in
C^1_0([0,\infty)\times\R)$, one has
\begin{align}
  \int_0^\infty \int_\R \big(\abs{\rho-k}\test_t + \sign{\rho-k}&(f(\rho)-f(k))\test_z\big) \,dzdt  \notag\\
 &\qquad  +
  \int_\R \abs{\rho(0,z)-k} \test(0,z)\,dz \ge 0.   \label{eq:scalent}
\end{align}

More precisely, we show the following result. Assume that the velocity function satisfies
the reasonable assumptions \eqref{eq:vassume}, 
and the initial data $\rho(0,\dott)\in L^1(\R)\cap BV(\R)$.
  Let $\rho_\ell(t,z)$ be the density of vehicles as defined by the FtL model, see \eqref{eq:ellrhoV}. Then we show that $\lim_{\ell\to 0}\rho_\ell
  = \rho\in C([0,\infty);L^1(\R))$, where $\rho$ is the unique 
  solution to \eqref{eq:scalcons} satisfying the entropy condition \eqref{eq:scalent} such that $\rho(0,z)=\rho_0(z)$.

The rest of this note is organized as follows:  In
Section~\ref{sec:model} we define the discrete model and prove some
simple bounds on its solutions, and in Section~\ref{seq:proof} we give
the elementary proof of convergence.

\section{The model}\label{sec:model}
We use units such that $v_{\max}=1$.
Let $v(\rho)$ be a continuously differentiable function $v\colon[0,1]\to
[0,1]$, such that $v'\le 0$, $v(0)=1$ and $v(1)=0$. We use the
notation $V(y)=v(1/y)$, assume that 
\begin{subequations}\label{eq:vassume}
  \begin{align}
    \label{eq:vdecrease}
    V(y)&\ge 1 - \frac{1}{y^{\sigma-1}}, \quad \text{for some 
      constant $\sigma>1$, }\\
    y^2 V'(y) &\le M, \quad \text{for $y\ge 1$, and for some constant
      $M$.}
    \label{eq:vprimebound}
  \end{align}
\end{subequations}
Define the forward
difference 
\begin{equation*}
  \Dp h_i = \frac{1}{\ell}\big(h_{i+1}-h_i \big).
\end{equation*}
Let $\seq{y_i(t)}_{i=1}^{N-1}$ satisfy
\begin{equation}
  \label{eq:ydt}
  \dot{y}_i = \Dp V_i, \ \ i=1,\ldots,N-1,\ \   \ t>0,
\end{equation}
where $V_i=V(y_i)$ and $V_N = 1$. Later we will also need $y_N = \infty$. Regarding the initial values, we assume that there
is a function $\rho_0\colon\R \to [0,1]$ normalized such that $\int_\R
\rho_0 \,dz = 1$.  Define $\seq{z_{i+1/2}(0)}_{i=0}^{N-1}$ inductively as
\begin{equation}\label{eq:zinit}
   \int_{z_{i-1/2}(0)}^{z_{i+1/2}(0)} \rho_0(z)\,dz =
   \frac{1}{N+1} = \ell,\quad i=0,\ldots,N-1.
\end{equation}
Thus with the current scaling, the length $\ell$ of each vehicle is $\ell=1/(N+1)$.
We will also need $z_{-1/2}(0)=-\infty$.
Here we choose the infimum of possible values for $z_{i+1/2}(0)$ satisfying \eqref{eq:zinit}. 
Set 
\begin{equation}\label{eq:yinit}
  y_i(0)=\frac{1}{\ell}\big(z_{i+1/2}(0)-z_{i-1/2}(0)\big),\quad i=1,\ldots,N-1.
\end{equation}
Observe that it follows from \eqref{eq:zinit} that  $y_i(0)\ge 1$ for $i=1,\ldots,N-1$, since $\rho_0\in[0,1]$.
\begin{lemma}
  \label{lem:yellbnd}
  Assume that $V$ satisfies \eqref{eq:vdecrease} and that
  $\seq{y_i}_{i=1}^{N-1}$ solves the system \eqref{eq:ydt} with
  initial values \eqref{eq:yinit}. Then
  \begin{equation}\label{eq:yjbnd}
    1\le y_i(t)\le \Bigl(y_i(0)^\sigma + \frac{\sigma
      t}{\ell}\Bigr)^{1/\sigma},\quad
    i=1,\ldots,N-1.
  \end{equation}
  In particular,
  \begin{equation}\label{eq:yelllim}
    \lim_{\ell\to 0} \left(\ell^\kappa y_{i}(t)\right) = 0, \, t\in
      (0,\infty), \,\kappa>1/\sigma,\, i=1,\dots,N-1.
  \end{equation}
\end{lemma}
\begin{proof}
  If $y_i(t)=1$, then $V(y_i(t))=0$ and hence $\dot{y}_i(t)\ge
  0$. This gives the lower bound on $y_i$.

  Using \eqref{eq:vdecrease} and the bound $V_{i+1}\le 1$,  we get
  \begin{equation*}
    \dot{y}_i = \frac{1}{\ell}\left(V_{i+1}-V_i\right) \le
    \frac{1}{\ell y_i^{\sigma-1}}. 
  \end{equation*}
 By integrating this inequality, 
we see that the estimate \eqref{eq:yjbnd} holds, and the limit
  \eqref{eq:yelllim} then follows trivially.
\end{proof}

\begin{lemma}
  \label{lem:vrhobv}
  Define $\rho_i(t)=1/y_i(t)$.  Write $V_i(t)=V(y_i(t))$. 
  We have that 
  \begin{align}
     \sum_{i=1}^{N-1} \abs{V_{i+1}(t)-V_i(t)} &\le \sum_{i=1}^{N-1}
      \abs{V_{i+1}(0)-V_i(0)} \label{eq:BVV}\\
       \intertext{and} 
       \sum_{i=1}^{N-1}
      \abs{\rho_{i+1}(t)-\rho_i(t)} &\le  \sum_{i=1}^{N-1}
      \abs{\rho_{i+1}(0)-\rho_i(0)}.\label{eq:BVrho}
  \end{align}
\end{lemma}
\begin{proof}
  We find that
  \begin{align*}
    \frac{d}{dt}\abs{V_{i+1}-V_i}&=\sign{V_{i+1}-V_i} V'(y_{i+1})\Dp
    V_{i+1} -  V'(y_{i}) \abs{\Dp V_i}\\
    &\le V'(y_{i+1})\abs{\Dp V_{i+1}} - V'(y_{i}) \abs{\Dp V_i},
  \end{align*}
  since $V'\ge 0$.
  For $i=N$ we recall the conventions that $y_N=\infty$ and $V_N=1$, and
  thus $V'(y_N)=0$.  We have that $0\le V_i \le 1$. Then
  \begin{align*}
    \frac{d}{dt}\abs{V_{i+1}-V_i}&=\sign{V_{i+1}-V_i} V'(y_{i+1})\Dp
    V_{i+1} -  V'(y_{i}) \abs{\Dp V_i}\\
    &\le V'(y_{i+1})\abs{\Dp V_{i+1}} - V'(y_{i}) \abs{\Dp V_i},
  \end{align*}
  since $V'\ge 0$. This means that
  \begin{align}
    \label{eq:Vbv}
    \frac{d}{dt}\sum_{i=1}^{N-1} \abs{V_{i+1}-V_i} &\le
    \frac{1}{\ell}\left(V'(y_N)\abs{1-1}
      - V'(y_1) \abs{V_2-V_1}\right)\\
    &=-\frac{1}{\ell} V'(y_1) \abs{V_2-V_1} \le 0,
  \end{align}
  which shows \eqref{eq:BVV}.
  Since $\sign{V_{i+1}-V_i}=-\sign{\rho_{i+1}-\rho_i}$, we could also
  carry out these estimates for $\rho_i$, proving \eqref{eq:BVrho}.
\end{proof}
Note that 
\begin{equation*}
  \sum_i \abs{\rho_{i+1}(0)-\rho_i(0)}\le \abs{\rho_0}_{BV},
\end{equation*}
where $\abs{\dott}_{BV}$ denotes the bounded variation norm.
For $t>0$, define 
$z_{i+1/2}(t)$ by
\begin{equation}\label{eq:zdt}
  \dot{z}_{i-1/2} = V_{i},\quad  i=1,\ldots N, 
\end{equation}
with initial values given by \eqref{eq:zinit}. Combining \eqref{eq:ydt},  \eqref{eq:yinit}, and  \eqref{eq:zdt}
we conclude that
\begin{equation}\label{eq:yzinit}
  y_i(t)=\frac{1}{\ell}\big(z_{i+1/2}(t)-z_{i-1/2}(t)\big),\quad t\in[0,\infty), \, i=1,\ldots,N-1.
\end{equation}
In particular, 
\begin{equation}\label{eq:zrho}
\big(z_{i+1/2}(t)-z_{i-1/2}(t)\big) \rho_i(t)=\ell, \quad i=1,\ldots,N-1.
\end{equation}
Note that $z_{i-1/2}$
coincides with the position of the $i$th vehicle from the left, given
by \eqref{eq:Zeq}.  Thus $Z_i=z_{i+1/2}$. 

Furthermore, define the functions 
\begin{equation} \label{eq:ellrhoV}
\begin{aligned}
  \rho_\ell(t,z)&=\sum_{i=1}^{N-1}\rho_i(t) \cha{[z_{i-1/2}(t),z_{i+1/2}(t))}(z),\\
  V_\ell(t,z)&=\sum_{i=1}^{N-1}V_i(t) \cha{[z_{i-1/2}(t),z_{i+1/2}(t))}(z),
\end{aligned}
\end{equation}
where $\cha{I}$ denotes the characteristic function of an interval $I$. Observe that Lemma~\ref{lem:vrhobv} implies that
\begin{equation} \label{eq:BVt}
  \abs{\rho_\ell(t)}_{BV}\le  \abs{\rho_\ell(0)}_{BV}, \quad   \abs{V_\ell(t)}_{BV}\le  \abs{V_\ell(0)}_{BV}. 
\end{equation}

\section{The continuum limit}\label{seq:proof}
\begin{theorem}
  \label{thm:ourtheorem}
  Assume that the function $V$ satisfies \eqref{eq:vassume}, and
  that $\rho_0\in L^1(\R)\cap BV(\R)$.
  Let $\rho_\ell$ be as defined above. Then $\lim_{\ell\to 0}\rho_\ell
  = \rho\in C([0,\infty);L^1(\R))$, where $\rho$ is the unique entropy
  solution to \eqref{eq:scalcons} such that $\rho(0,z)=\rho_0(z)$.
\end{theorem}
\begin{proof}
Observe that $\rho_\ell$ is in $L^1(\R)$, since it is positive and 
\begin{equation*}
  \frac{d}{dt}\int_\R \rho_\ell(t,z) \,dz = \sum_{i=1}^{N-1} \frac{d}{dt}
  \int_{z_{i-1/2}}^{z_{i+1/2}} \rho_i\,dz =
  \sum_{i=1}^{N-1} \frac{d}{dt} \big((z_{i+1/2}-z_{i-1/2})\rho_i\big) = 0.
\end{equation*}
Hence $\norm{\rho_\ell(t)}_{L^1}\le 1$.
For any $\seq{h_i}$ define 
\begin{equation*}
  \Dp^z h_i = \frac{h_{i+1}-h_i}{z_{i+1/2}-z_{i-1/2}}=\rho_i \Dp h_i.
\end{equation*}
Let $\test=\test(t,z)$ be a smooth test function with compact support in
$\R\times (0,\infty)$. We calculate
\begin{align}
  \int_0^\infty \int_\R \rho_\ell \test_t\,dzdt &= \int_0^\infty
  \sum_i \rho_i \int_{z_{i-1/2}}^{z_{i+1/2}} \test_t
  \,dz\,dt \notag\\
  &=\int_0^\infty \sum_i \Big[\rho_i\,\frac{\partial}{\partial t}\Bigl( \int_{z_{i-1/2}}^{z_{i+1/2}} \test
  \,dz\Bigr) - \rho_i \ell \Dp\left(\dot{z}_{i-1/2}\test_{i-1/2}\right) \Big]\,dt
   \notag\\
  &=-\int_0^\infty\sum_i \Big[ \int_{z_{i-1/2}}^{z_{i+1/2}} \dot{\rho}_i
  \test \, dz - \ell\Dp (\rho_i) V_{i+1} \test_{i+1/2}\Big] \,dt \notag\\
  &= \int_0^\infty \sum_i\int_{z_{i-1/2}}^{z_{i+1/2}}\Big( \rho_i^2 \Dp(V_i) \test + \Dp^z
  (\rho_i) V_{i+1} \test_{i+1/2}\Big)\,dz \,dt \notag\\
  &= \int_0^\infty \sum_i\int_{z_{i-1/2}}^{z_{i+1/2}}\Big( \rho_i \Dp^z(V_i) \test + \Dp^z
  (\rho_i) V_{i+1} \test_{i+1/2}\Big)\,dz \,dt,\label{eq:rhotweak}
\end{align}
where we have used that $\dot{\rho}_i = - \rho_i^2\Dp V_i$, and introduced  the
notation $\test_{i+1/2}=\test(t,z_{i+1/2})$.
Similarly,
\begin{align}
  \int_0^\infty \int_\R \rho_\ell V_\ell \test_z\,dzdt &=
  \int_0^\infty \sum_i \rho_i V_i \int_{z_{i-1/2}}^{z_{i+1/2}} \test_z
  \,dz\,dt \notag \\
  &=-\int_0^\infty \sum_i \ell\Dp\left(\rho_iV_i\right)\test_{i+1/2}
  \,dt \notag \\
  &=-\int_0^\infty \sum_i \int_{z_{i-1/2}}^{z_{i+1/2}}\left( \rho_i
  \Dp^z\left(V_i\right) + \Dp^z\left(\rho_i\right)V_{i+1}\right)
  \test_{i+1/2}\,dz\,dt.   
  \label{eq:rhoxweak}
\end{align}
Using \eqref{eq:rhotweak} with $\test(t,z)=\cha{[t_1,t_2]}(t)\psi(z)$ for $0<t_1<t_2<\infty$ and a smooth test function $\psi$ with $\abs{\psi}\le1$, we formally
get with $\psi_{i+1/2}=\psi(z_{i+1/2})$ that
\begin{align*}
\Big| \int_\R \big(\rho_\ell(t_2,z)&- \rho_\ell(t_1,z)\big)\psi(z)\,dz\Big| \\
&= \Big|\int_{t_1}^{t_2} \sum_i\int_{z_{i-1/2}}^{z_{i+1/2}}\Big( \rho_i \Dp^z(V_i) \psi + \Dp^z
  (\rho_i) V_{i+1} \psi_{i+1/2}\Big)\,dz \,dt \Big| \\
&\le  \int_{t_1}^{t_2} \sum_i\int_{z_{i-1/2}}^{z_{i+1/2}}\Big( \abs{\rho_i} \abs{\Dp^z(V_i)} \abs{\psi} + \abs{\Dp^z
  (\rho_i)} \abs{V_{i+1}} \abs{\psi_{i+1/2}}\Big)\,dz \,dt  \\
&\le  \ell\int_{t_1}^{t_2} \sum_i\Big( \abs{\rho_i} \abs{\Dp(V_i)}  + \abs{\Dp
  (\rho_i)} \abs{V_{i+1}} \Big)\,dt  \\ 
&\le  \ell\int_{t_1}^{t_2} \sum_i\Big( \abs{\Dp(V_i)}  + \abs{\Dp
  (\rho_i)} \Big) \,dt  \\  
  &\le(t_2-t_1) \sum_i\big(\abs{V_{i+1}(0)-V_{i}(0)}+\abs{\rho_{i+1}(0)-\rho_{i}(0)} \big)\\
  &\le (t_2-t_1)\big(\abs{V_\ell(0)}_{BV}+\abs{\rho_\ell(0)}_{BV} \big), 
\end{align*}
using first that $\abs{\psi}, \abs{\rho_i}, \abs{V_i}\le 1$ and subsequently Lemma \ref{lem:vrhobv}.  By approximating the
characteristic function $\cha{[t_1,t_2]}$ with a smooth function, and taking the limit, we still obtain the above estimate.
This implies
\begin{align*}
  \norm{\rho_\ell(t_2)-\rho_\ell(t_1)}_{L^1} &=
  \sup_{\abs{\psi}\le 1} \int
  (\rho_\ell(t_2,z)-\rho_\ell(t_1,z))\psi(z)\,dz\\
  &\le (t_2 - t_1)\big(\abs{\rho_\ell(0)}_{BV}+\abs{V_\ell(0)}_{BV}\big).  
\end{align*}
Thus, recalling \eqref{eq:BVt}, we can apply \cite[Theorem A.11]{HoldenRisebro} to conclude that the set
$\seq{\rho_\ell}_{\ell>0}$ is compact in  $C([0,\infty);L^1(\R))$, and
there exists a sequence $\seq{\ell_j}_{j=1}^\infty$, $\ell_j\to 0$ as
$j\to \infty$, and a function $\rho$ such that 
\begin{equation*}
  \rho_{\ell_j} \to \rho \quad \text{in $C([0,\infty);L^1(\R))$, as $j\to \infty$.}
\end{equation*}
To simplify the notation, we henceforth write $\ell=\ell_j$.
Furthermore,  since $v(\rho_\ell)=V_\ell$, $V_\ell \to v(\rho)$. Adding
\eqref{eq:rhotweak} and \eqref{eq:rhoxweak}, we get
\begin{align*}
  \Bigl| \int_0^\infty\int_\R \big( \rho_\ell \test_t + \rho_\ell V_\ell
  \test_z\big)\,dzdt \Bigr| &=
  \Bigl| \int_0^\infty \sum_i \rho_i (V_{i+1}-V_i) \\
  &\qquad\qquad\qquad \times \int_{z_{i-1/2}}^{z_{i+1/2}}
  \left(\test(t,z)-\test(t,z_{i+1/2})\right)\, dz\,dt\Bigr|\\
  &\le \frac12   \int_0^\infty  \sup_{i}(\ell y_i(t))^2 
  \norm{\test_z(t,\dott)}_{L^\infty}\sum_i
  \abs{V_{i+1}-V_i}\,dt \\
  &\to 0,\ \ \text{as $\ell$ to 0,}
\end{align*}
and thus $\rho$ is a weak solution.
To show that $\rho$ is an entropy solution,
let $\eta$ be a  twice differentiable convex function. Since
$V'\ge 0$, we get
\begin{align*}
  \frac{d}{dt}\eta(y_i)&=\eta'(y_i)\Dp V_i=\frac1\ell\eta'(y_i)\int_{y_i}^{y_{i+1}}V'(y) \, dy\\
&\le \frac1\ell\int_{y_i}^{y_{i+1}}\eta'(y)V'(y) \, dy  = \frac1\ell\int_{y_i}^{y_{i+1}}Q'(y) \, dy 
    =\Dp Q_i,
\end{align*}
where $Q'=\eta' V'$ and $Q_i=Q(y_i)$. Introduce $q(\rho)=Q(1/\rho)$ with $q_i=q(\rho_i)$. Define $\mu=\mu(\rho)$ by
$\mu(\rho)=\rho\eta(1/\rho)$.  As usual we write $\mu_i=\mu(\rho_i)$ and $\eta_i=\eta(y_i)$. Then $\mu$ is a convex function of
$\rho$, and if $\mu$ is a convex function of $\rho$, then $\eta$ is a
convex function of $y$. We have that
\begin{equation*}
   \frac{d}{dt} \mu(\rho_i) = -\rho_i^2\left(\Dp V_i)\right) \eta_i +
   \rho_i \frac{d}{dt} \eta_i.  
\end{equation*}
Set $\mu_\ell(t,z) = \sum_i \mu_i(t) \cha{[z_{i-1/2}(t),z_{i+1/2}(t))}(z)$ with $\mu_i(t)=\mu(\rho_i(t))$, and define $q_\ell(t,z)$
similarly. As when establishing \eqref{eq:rhotweak}, we find  for a non-negative test function $\test$ with support in $\R\times(0,\infty)$ that
\begin{align*}
  \int_0^\infty \int_\R& \mu_\ell \test_t\,dzdt \\
  &= \int_0^\infty \sum_i
  \mu_i \int_{z_{i-1/2}}^{z_{i+1/2}} \test_t
  \,dz\,dt \\
  &=\int_0^\infty \sum_i\Big[ \mu_i \,\frac{\partial}{\partial
    t}\Bigl(\int_{z_{i-1/2}}^{z_{i+1/2}} \test \,dz\Bigr) - \mu_i \ell
  \Dp\left(\dot{z}_{i-1/2}\test_{i-1/2}\right)\Big] \,dt
  \\
  &=-\int_0^\infty \sum_i \Big[ \int_{z_{i-1/2}}^{z_{i+1/2}} \dot{\mu}_i
  \test \, dz - \ell\Dp (\mu_i) V_{i+1} \test_{i+1/2}\Big]  \,dt\\
  &\ge \int_0^\infty \sum_i\int_{z_{i-1/2}}^{z_{i+1/2}} \Big[\left(\eta_i
    \rho_i^2 \Dp(V_i) - \rho_i \Dp(q_i) \right)\test + \Dp^z
  (\mu_i) V_{i+1} \test_{i+1/2}\Big] \,dz \,dt\\
  &= \int_0^\infty \sum_i\int_{z_{i-1/2}}^{z_{i+1/2}}\Big[ \left(\mu_i
    \Dp^z(V_i) - \Dp^z(q_i)\right)\test + \Dp^z (\mu_i) V_{i+1}
  \test_{i+1/2}\Big]\,dz \,dt.
\end{align*}
Similarly
\begin{align*}
  \int_0^\infty \int_\R& \left(V_\ell\mu_\ell - q_\ell\right)\test_z\,dzdt \\
  &=
  \int_0^\infty \sum_i\left(V_i\mu_i -q_i\right)
  \int_{z_{i-1/2}}^{z_{i+1/2}}\test_z \,dz \,dt\\
  &=-\int_0^\infty \sum_i \ell\big[\Dp(\mu_iV_i) - \Dp
    (q_i)\big]\test_{i+1/2}\,dt\\
  &=-\int_0^\infty\sum_i \int_{z_{i-1/2}}^{z_{i+1/2}} \big[\mu_i \Dp^z(V_i) +
  V_{i+1}\Dp^z(\mu_i) - \Dp^z(q_i)\big]\test_{i+1/2}\,dz\,dt.
\end{align*}
Therefore, 
\begin{equation*}
  \int_0^\infty\int_\R \big(\mu_\ell \test_t + \left( \mu_\ell
    V_\ell-q_\ell\right) \test_z\big) \,dzdt \ge r_\ell,
\end{equation*}
where 
\begin{align*}
  \abs{r_\ell} &=\Bigl| \int_0^T \sum_i \int_{z_{i-1/2}}^{z_{i+1/2}}
  \left( \mu_i \Dp^z(V_i) - \Dp^z(q_i)\right) \left(\test -
    \test_{i+1/2}\right) \,dz\,dt\Bigr|\\
  &\le \int_0^\infty \sum_i\left( \abs{\mu_i}\,\abs{V_{i+1}-V_i} +
    \abs{q_{i+1}-q_i}\right) \int_{z_{i-1/2}}^{z_{i+1/2}}
  \abs{\test(t,z)-\test(t,z_{i+1/2})}\,dz\,dt\\
  &\le C   \int_0^\infty  \sup_{i}(\ell y_i(t))^2 
  \norm{\test_z(t,\cdot)}_{L^\infty}\sum_i\big(
  \abs{V_{i+1}-V_i}+\abs{q_{i+1}-q_i}\big)\,dt.
\end{align*}
If $q_\ell(t,\dott)$ is of bounded variation, then $r_\ell \to
0$. We now assume that $\mu$ is (a smooth approximation to) the
Kru\v{z}kov entropy $\mu(\rho)=\abs{\rho-k}$. A short computation yields that
\begin{equation*}
\mu(\rho) V\big(\frac1\rho\big)-q(\rho)=  \sign{\rho-k}\Big(\rho V(\frac1\rho)-kV(\frac1k)\Big), 
\end{equation*}
which is consistent with \eqref{eq:scalent}.
Then
$\eta(y)=y\abs{1/y-k}$, and $\abs{\eta'(y)}\le \abs{k}$.
If $V$ satisfies \eqref{eq:vprimebound}, the mapping $\rho\mapsto
q(\rho)$ is Lipschitz, since
\begin{equation*}
  \abs{q'(\rho)}= \abs{\frac{d}{d\rho}Q\left(\frac1\rho\right)}
  =\Bigl|\eta'\left(\frac{1}{\rho}\right)V'\left(\frac{1}{\rho}\right)\frac{1}{\rho^2}
  \Bigr|
  \le M \abs{k}.
\end{equation*} 
Hence  $q_\ell(t,\dott)$ is of bounded variation,
since $\rho_\ell$ is in $BV$.

A similar argument with a test function whose support include the initial data on $t=0$, will show
\eqref{eq:scalent}. We conclude that $\rho$ is an entropy solution. Since the entropy
solution is unique, we also conclude that the whole sequence, rather
than just a subsequence, converges.
\end{proof}

\end{document}